\documentclass[12pt]{amsart}
\usepackage{amssymb}
\newtheorem{defin}{Definition}
\newtheorem{propo}{ Proposition}
\newtheorem{lemme}{Lemma}
\newtheorem{coro}{Corollary}
\newtheorem{theorem}{Theorem}

\newtheorem{exe}{Example}

\begin{document}

\title [ Dunkl positive definite functions]{ Dunkl positive definite functions \\}%

\author[ J. El Kamel, K. Mehrez]{ Jamel El Kamel \qquad and \qquad Khaled Mehrez }
 \address{Jamel El Kamel. D\'epartement de Math\'ematiques fsm. Monastir 5000, Tunisia.}
 \email{jamel.elkamel@fsm.rnu.tn}
 \address{Khaled Mehrez. D\'epartement de Math\'ematiques IPEIM. Monastir 5000, Tunisia.}
 \email{k.mehrez@yahoo.fr}
\begin{abstract}
We introduce the notion of Dunkl positive definite and strictly positive definite functions on $\mathbb{R}^{d}$. This done by the use of the properties of Dunkl translation. We establish the analogue of Bochner's theorem in Dunkl setting. The case of radial functions is considered. We give a sufficient condition for a function to be Dunkl strictly positive definite on $\mathbb{R}^{d}.$  
\end{abstract}
\maketitle
{\it keywords:} Positive definite functions; Dunkl transform; 
Dunkl translation.\\
MSC (2000) 42A38, 42B08, 42815, 33D15, 47A05\\
\section{Introduction}
In classical analysis a complex valued continuous function is said positive definite (resp. strictly positive definite) on $\mathbb{R}$, if for every distinct real numbers $x_{1},x_{2},...,x_{n}$ and every complex numbers $z_{1},z_{2},...,z_{n}$ not all zero, the inequality 
\[\sum_{j=1}^{n}\sum_{k=1}^{n}z_{j}\overline{z_{k}}f(x_{j}-x_{k})\geq0\,\,(resp.\,>0)\]
hold true. (\rm{see [7]})

In 1930, the class of positive definite functions is fully characterized by Bochner's theorem [1], the function $f$ being positive definite if and only if it is the Fourier transform of a nonnegative finite Borel measure on the real line $\mathbb{R}.$

In this work, we introduce the analogue of positive definite function in Dunkl setting. This done by the use of the properties of the Dunkl translation. We establish a version of Bochner's theorem in Dunkl setting. We give a sufficient condition for a function in $A_{\kappa}(\mathbb{R}^{d})$ to be Dunkl strictly positive definite.

Our paper is organized as follows: In section 2, we present some preliminaries results and notations that will be useful in the sequel. In section 3, we give some properties of the Dunkl transform, the Dunkl translation and the Dunkl convolution. In section 4, we introduce the notion of the Dunkl positive definite functions in studying their properties, some examples are given. We prove that if $\varphi\in A_{\kappa}(\mathbb{R}^{d})$ is Dunkl positive definite, then the Dunkl transform of $\varphi$ is nonnegative and $\varphi$ is bounded. The case of radial function is considered. We state a version of Bochner's theorem in Dunkl setting. As application, we are interested with the Dunkl heat kernel , and we get a new equality for the modified Bessel function. The section 5 is devoted to Dunkl strictly positive definite functions.
\section{Notations and preliminaries}
Let $R$ be a fixed root system in $\mathbb{R}^{d}$, $G$ the associated finite reflexion group, and $R_{+}$ a fixed positive subsystem of $R,$ normalized so that $<\alpha,\alpha>=2$ for all $\alpha\in R_{+}$, where $<x,y>$ denotes the usual Euclidean inner product.

For a non zero $\alpha\in \mathbb{R}^{d}$, let use define the reflexion $\sigma_{\alpha}$ by 
\[\sigma_{\alpha}x=x-2\frac{<x,\alpha>}{<\alpha,\alpha>}\alpha,\,\, x\in\mathbb{R}^{d}.\]
Let $\kappa$ be a nonnegative multiplicity function $\alpha\longmapsto\kappa_{\alpha}$ defined on $R_{+}$ with the property that $\kappa_{\alpha}=\kappa_{\beta}$ where $\sigma_{\alpha}$ is conjugate to $\sigma_{\beta}$ in $G$. The weight function $h_{\kappa}$ est defined by 
\begin{equation}
h_{\kappa}(x)=\prod_{\alpha\in R_{+}}|<x,\alpha>|^{\kappa_{\alpha}},\,\, x\in\mathbb{R}^{d}.
\end{equation}
This is a nonegative homogeneous function of degre $\displaystyle{\gamma_{\kappa}=\sum_{\alpha\in R_{+}}\kappa_{\alpha}}$, which is invariant under the reflexion group $G.$\\
Let $T_{i}$ denote Dunkl's differential-difference operator defined in [2] by 
\begin{equation}
T_{i}f(x)=\partial_{i}f(x)+\sum_{\alpha\in R_{+}}\kappa_{\alpha}\frac{f(x)-f(\sigma_{\alpha}x)}{<\alpha,x>}<\alpha,e_{i}>,\,\,1\leq i\leq d,
\end{equation}
where $\partial_{i}$ is the ordinary partial derivative with respect to $x_{i}$, and $e_{1},e_{2},...,e_{d}$ are the standard unit vectors of $\mathbb{R}^{d}.$ It was proved in [2] that $T_{1},T_{2},...,T_{d}$ commute. Therefore we can naturally define $P(T)$ for any polynomial $P$, where $T=(T_{1},T_{2},...,T_{d}).$\\
Let $\mathcal{P}_{n}^{d}$ denote the space  of homogeneous polynomials of degree $n$ in $d-$variables. The operators $T_{i},\,1\leq i\leq d$ map  $\mathcal{P}_{n}^{d}$ to $\mathcal{P}_{n-1}^{d}.$ The intertwinig operator $V_{\kappa}$ is linear and determined uniquely as 
\begin{equation}
V_{\kappa}\mathcal{P}_{n}^{d}\subset\mathcal{P}_{n}^{d},\,\,V_{\kappa}1=1,\,\,\mathcal{T}_{i}V_{\kappa}=V_{\kappa}\partial_{i},\,1\leq i\leq d.
\end{equation}
The Dunkl kernel $E_{\kappa}$ associated with $G$ and $\kappa$ is defined by 
\begin{equation}
E_{\kappa}(x,y)=V_{\kappa}\left(e^{<.,y>}\right)(x),\,\, x,y\in\mathbb{R}^{d}.
\end{equation}
\begin{propo}(\rm{see[8]})
Let $y\in\mathbb{C}^{d}$. Then the function $f=E_{\kappa}(.,y)$
is the unique solution of the system 
\begin{equation}
T_{i}f=<e_{i},y>f,\, \textrm{for\, all}\,\,1\leq i\leq d,
\end{equation}
which is real-analytic in $\mathbb{R}^{d}$ and satisfies $f(0)=1.$
\end{propo}
We collect some further properties of the Dunkl kernel $E_{\kappa}$
\begin{propo}(\rm{see[4],[8]})
For $x,y\,\in\mathbb{C}^{d},\,\lambda\in\mathbb{C}$
\begin{enumerate}
\item $E_{\kappa}\left(x,\, y\right)=E_{\kappa}\left(y,\, x\right),$
\item $E_{\kappa}\left(\lambda x,\, y\right)=E_{\kappa}\left(x,\, \lambda y\right)$
\item $\overline{E_{\kappa}\left(x,\, y\right)}=E_{\kappa}\left(\overline{x},\,\overline{y}\right)$
\item $|E_{\kappa}(-ix,y)|\leq1$.
\item $\mid E_{\kappa}\left(x,\, y\right)\mid\leq e^{\parallel x\parallel.\parallel y\parallel},$
\item Let $\nu(z)=z_{1}^{2}+...+z_{d}^{2},\, z_{i}\in\mathbb{C}.$ For $z,\,\omega\in\mathbb{C}^{d},$
\[
c_{\kappa}\int_{\mathbb{R}^{d}}E_{\kappa}\left(x,\, z\right)E_{\kappa}\left(x,\,\omega\right)h_{\kappa}^{2}\left(x\right)e^{-\frac{\parallel x\parallel^{2}}{2}}dx=e^{\frac{\left(\nu(z)+\nu(\omega)\right)}{2}}E_{\kappa}\left(z,\,\omega\right),\]
\end{enumerate}
where $c_{\kappa}$ denotes the Mehta-type costant  defined by 
\begin{equation}
c_{\kappa}^{-1}=\int_{\mathbb{R}^{d}}h_{\kappa}^{2}(x)e^{-\frac{\parallel x\parallel^{2}}{2}}dx.
\end{equation}
\end{propo} 
In particular, the function 
$$E_{\kappa}(x,y)=V_{\kappa}^{(x)}\left(e^{<x,y>}\right),\,\, x,y\in\mathbb{R}^{d},$$
plays the role of $e^{i<x,y>}$ in the ordinary Fourier analysis. 
Trought this paper, we fix the values of $\gamma$ and $\lambda$ as 
\begin{equation}
\gamma:=\gamma_{\kappa}=\sum_{\alpha\in R_{+}}k_{\alpha}\,\,\, \textrm{and}\,\,\,\lambda:=\gamma+\frac{d-2}{2}.
\end{equation}
Let us recall some classical functional spaces: 
\begin{itemize} 
\item $C(\mathbb{R}^{d})$ the set of continuous functions on $\mathbb{R}^{d}$ and $C_{0}(\mathbb{R}^{d})$ its subspace of continuous functions on $\mathbb{R}^{d}$ vanishing at infnity.

\item $S(\mathbb{R}^{d})$ the Schwartz space of infnitely differentiable functions on $\mathbb{R}^{d}$ which are rapidly decreasing as their derivatives.

\item $L^{p}\left(\mathbb{R}^{d},h_{\kappa}^{2}\right),\;1\leq p<\infty$, the space of measurable functions on $\mathbb{R}^{d}$ such that 
\[\parallel f\parallel_{\kappa,p}=\left(\int_{\mathbb{R}^{d}}|f(x)|^{p}h_{\kappa}^{2}(x)dx\right)^{\frac{1}{p}}<\infty.\]
\end{itemize}

\section{Harmonic analysis related to the Dunkl operator}
In this section, we present some properties of the Dunkl transform, the Dunkl translation and the Dunkl convolution studied and developed in great detail in [4,6,10,11].\\
The Dunkl transform is defined for $f\in L^{1}\left(\mathbb{R}^{d},h_{\kappa}^{2}\right)$ by 
\begin{equation}
D_{\kappa}f(x)=c_{\kappa}\int_{\mathbb{R}^{d}}f(y)E_{\kappa}\left(-ix,\, y\right)h_{\kappa}^{2}(y)dy,\,\,\, x\in\mathbb{R}^{d}.
\end{equation}
If $\kappa=0$, then $V_{\kappa}=id$ and the Dunkl transform coincides with the usual Fourier transform. If $d=1$ and $G=\mathbb{Z}_{2}$, then the Dunkl transform is related closely to the Hankel transform on the real line.\\
In fact, in this case,
\[
E_{\kappa}\left(x,\,-iy\right)=\Gamma\left(\kappa+\frac{1}{2}\right)\left(\frac{\mid xy\mid}{2}\right)^{-\kappa+\frac{1}{2}}\left[J_{\kappa-\frac{1}{2}}(\mid xy\mid)-i\;\textrm{sign}(xy)\, J_{\kappa+\frac{1}{2}}(\mid xy\mid)\right],\]
where $J_{\alpha}$ denotes the usual Bessel function of first kind and order $\alpha.$
\begin{theorem}(\rm{see [10]})
\begin{enumerate}
\item For $f\in L^{1}\left(\mathbb{R}^{d},\, h_{k}^{2}\right),$ we have $D_{\kappa}f\in C_{0}\left(\mathbb{R}^{d}\right),$ and 
\[\parallel D_{\kappa}f\parallel_{C_{0}}\leq\parallel f\parallel_{\kappa,1}.\]
\item When both $f$ and $D_{\kappa}f\in L^{1}\left(\mathbb{R}^{d},\, h_{k}^{2}\right),$ we have the inversion formula 
\[f(x)=c_{\kappa}\int_{\mathbb{R}^{d}}D_{\kappa}f(y)E_{\kappa}(ix,\, y)h_{\kappa}^{2}(y)dy \quad a.e.\] 
\item  The Dunkl transform $D_{\kappa}$ is an isomorphism of the Schwartz
class $\mathcal{S}(\mathbb{R}^{d})$ onto it self, and $D_{\kappa}^{2}f(x)=f(-x).$
\item The Dunkl transform $D_{\kappa}$ on $\mathcal{S}(\mathbb{R}^{d})$ extends uniquely to an isometry of $L^{2}\left(\mathbb{R}^{d},\, h_{k}^{2}\right).$
\item If $f,\, g\in L^{2}\left(\mathbb{R}^{d},\, h_{k}^{2}\right)$ then
\[\int_{\mathbb{R}^{d}}D_{\kappa}f(y)g(y)h_{\kappa}^{2}(y)dy=\int_{\mathbb{R}^{d}}f(y)D_{\kappa}g(y)h_{\kappa}^{2}(y)dy.\]
\end{enumerate}
\end{theorem}
Let $y\in\mathbb{R}^{d}$ be given. The Dunkl translation operator $f\longmapsto\tau_{y}f$ is defined in $L^{2}\left(\mathbb{R}^{d},\, h_{k}^{2}\right)$ by the equation 
\begin{equation}
D_{\kappa}(\tau_{y}f)(x)=E_{\kappa}(iy,x)D_{\kappa}f(x),\,\,\, x\in\mathbb{R}^{d}.
\end{equation}
The above definition gives $\tau_{y}f$ as an $L^{2}$ function.\\
Let 
\begin{equation}
\mathcal{A}_{\kappa}(\mathbb{R}^{d})=\left\{ f\in L^{1}(\mathbb{R}^{d},h_{\kappa}^{2}):\, D_{\kappa}f\in L^{1}(\mathbb{R}^{d},h_{\kappa}^{2})\right\}.
\end{equation}
Note that $\mathcal{A}_{\kappa}(\mathbb{R}^{d})$ is contained in the intersection of $L^{1}(\mathbb{R}^{d},h_{\kappa}^{2})$ and $L^{\infty}$ and hence is a subspace of $L^{2}(\mathbb{R}^{d},h_{\kappa}^{2})$. For $f\in A_{\kappa}(\mathbb{R}^{d})$ we have
\begin{equation}
\tau_{y}f(x)=\int_{\mathbb{R}^{d}}E_{\kappa}(ix,y)E_{\kappa}(-iy,\xi)D_{\kappa}f(\xi)h_{\kappa}^{2}(\xi)d\xi,\quad \forall x\in \mathbb{R}^{d}.
\end{equation}
\begin{theorem}(\rm{see [10]})
Assume that  $f\in\mathcal{A}_{\kappa}(\mathbb{R}^{d})$ and $g\in L^{1}(\mathbb{R}^{d},h_{\kappa}^{2})$ is bounded. Then 
\begin{enumerate}
\item $\displaystyle{\int_{\mathbb{R}^{d}}\tau_{y}f(\xi)g(\xi)h_{\kappa}^{2}(\xi)d\xi=\int_{\mathbb{R}^{d}}f(\xi)\tau_{-y}f(\xi)h(\xi)d\xi.}$
\vspace{0.4cm}

\item $\displaystyle{\tau_{y}f(x)=\tau_{-x}f(-y).}$
\end{enumerate}
\end{theorem}
\begin{theorem}(\rm{see[10]})
Let $f\in\mathcal{A}_{\kappa}\left(\mathbb{R}^{d}\right)$ be
a radial and nonnegative function. Then $T_{y}f\geq0$, $T_{y}f\in L_{\alpha}^{1}\left(\mathbb{R}^{d}\right)$ and
\begin{equation}
\int_{\mathbb{R}^{d}}T_{y}f(x)h_{\kappa}^{2}(x)dx=\int_{\mathbb{R}^{d}}f(x) h_{\kappa}^{2}(y)dx.
\end{equation}
\end{theorem}
The Dunkl convolution operator is defined on $L^{2}\left(\mathbb{R}^{d},\, h_{k}^{2}\right)$ by: for $f,\, g\in L^{2}\left(\mathbb{R}^{d},\, h_{k}^{2}\right)$,
\begin{equation}
f\star_{\kappa}g(x)=\int_{\mathbb{R}^{d}}f(y)\tau_{x}g^{\vee}(y)h_{\kappa}^{2}(y)dy,
\end{equation}
where $g^{\vee}(y)=g(-y).$\\
Note that as $\tau_{x}g^{\vee}\in L^{2}\left(\mathbb{R}^{d},\, h_{k}^{2}\right)$, the above convolution is well defined. We can also write the definition as 
\begin{equation}
f\star_{\kappa}g(x)=\int_{\mathbb{R}^{d}}D_{\kappa}f(\xi)D_{\kappa}g(\xi)E_{\kappa}(ix,\xi)h_{\kappa}^{2}(\xi)d\xi.
\end{equation}
\begin{theorem}(\rm{see [9,10,11]})
\begin{enumerate}
\item Let $f,g\in L^{2}\left(\mathbb{R}^{d},h_{\kappa}^{2}\right)$, then
\vspace{0.2cm}
\begin{enumerate}
\item $D_{\kappa}(f\star_{\kappa}g)=D_{\kappa}f.D_{\kappa}g.$
\vspace{0.2cm}
\item $f\star_{\kappa}g=g\star_{\kappa}f.$
\vspace{0.2cm}
\end{enumerate}
\item Let $f\in L^{2}\left(\mathbb{R}^{d},h_{\kappa}^{2}\right)$ and $g\in L^{1}\cap L^{2}\left(\mathbb{R}^{d},h_{\kappa}^{2}\right)$,
then $f\star_{\kappa}g\in L^{2}\left(\mathbb{R}^{d},h_{\kappa}^{2}\right)$
and \begin{equation}
\parallel f\star_{\kappa}g\parallel_{\kappa,2}\leq\parallel g\parallel_{\kappa,1}\parallel f\parallel_{\kappa,2}.
\end{equation}
\end{enumerate}
\end{theorem}

\section{Dunkl Positive definite Functions}
\begin{defin}\label{d1}
A continuous function $\varphi$ of $ L^{2}\left(\mathbb{R}^{d},\, h_{k}^{2}\right)$ is said 
Dunkl positive definite (resp. stictly Dunkl positive definte) if for  every finite distinct real numbers  $x_{1},...,x_{n},$ and every complex numbers $\alpha_{1}\,,...,\,\alpha_{n} $, not all zero, the inequality 
\[\sum_{j=1}^{n}\sum_{k=1}^{n}\alpha_{j}\overline{\alpha_{k}}\tau_{x_{j}}\left(\varphi\right)(x_{k})\geq0,\;\;(resp.>0)\]
holds true. 
Where $\tau_{x}$ denotes the Dunkl translation.
\end{defin}

From definition.\ref{d1} we can read of the elementary properties of a Dunkl positive definite function.
\begin{propo}\label{t1}(Properties of Dunkl positive definte functions)
\begin{enumerate}
\item A nonnegative finite linear combination of Dunkl positive definite functions is Dunkl positive definite.
\item Let $\varphi$ be a Dunkl positive definite function, then
\vspace{0.1cm}
\begin{enumerate}
\item The function $\tau_{x}\varphi(x)\geq0,$ for all $x\in\mathbb{R}^{d}$. In particular, $\varphi(0)\geq0.$
\vspace{0.1cm}
\item $\overline{\varphi(-x)}=\varphi(x)$, for all $x\in\mathbb{R}^{d}.$
\end{enumerate}
\end{enumerate}
\end{propo}
\begin{proof}
\begin{enumerate}
\item The first property is immediate consequence of the definition \ref{d1}.
\item The second property follows by choosing $n=1$, $\alpha_{1}=1$ and $x_{1}=x$ in the definition \ref{d1}.
\item In the definition \ref{d1}, let $n=2$, $x_1=0$, $\alpha_1=1$, $\alpha_2=c$ and $x_2=x$, then 

\[\varphi(0)+\mid c\mid^{2}\tau_{x}\varphi(x)+c\varphi(-x)+\overline{c}\varphi(x)\geq0.\]
Setting $c=1$ and $c=i$, respectively, we deduce that $\varphi(x)+\varphi(-x)$ and $i(\varphi(-x)-\varphi(x)$ must be reals. This can only be satisfied when $\overline{\varphi(-x)}=\varphi(x)$.
\end{enumerate}
\end{proof}
\begin{coro}\label{c11}
Let $\varphi\in \mathcal{A}_{\kappa}(\mathbb{R}^{d})$ be a Dunkl positive definite function, then $D_{\kappa}(\varphi)$
is real.
\end{coro}
\begin{proof}
For $\varphi\in \mathcal{A}_{\kappa}(\mathbb{R}^{d})$, we have 
\[D_{\kappa}(\varphi)(x)=c_{\kappa}\int_{\mathbb{R}^{d}}E_{\kappa}(y,-ix)\varphi(y)h_{\kappa}^{2}(y)dy.\]
Hence
\[\overline{ D_{\kappa}(\varphi)(x)}=c_{\kappa}\int_{\mathbb{R}^{d}}\overline{E_{\kappa}(y,-ix)}\; \overline{\varphi(y)}h_{\kappa}^{2}(y)dy.\]
Since $\overline{E_{\kappa}(x,y)}=E_{\kappa}(\overline{x},\overline{y})$ for $x,y\in \mathbb{C}^{d}$, we obtain
\begin{equation*}
\begin{split}
\overline{ D_{\kappa}(\varphi)(x)}&=c_{\kappa}\int_{\mathbb{R}^{d}}E_\kappa(y,ix)\overline{\varphi(y)}h_{\kappa}^{2}(y)dy\\
&=c_{\kappa}\int_{\mathbb{R}^{d}} E_{\kappa}(-y,ix)\; \overline{\varphi(-y)}h_{\kappa}^{2}(y)dy.
\end{split}
\end{equation*} 
So, by proposition \ref{t1}, we have 
$$\overline{\varphi(-x)}=\varphi(x),$$
and $E_{\kappa}(\lambda x,y)=E_{\kappa}(x,\lambda y)$, for any $\lambda\in\mathbb{C}$ we obtain 
$$\overline{ D_{\kappa}(\varphi)(x)}= D_{\kappa}(\varphi)(x).$$  
\end{proof}

\noindent We begin by seeking sufficient conditions for a function to be Dunkl positive definite. 
\begin{theorem}\label{t2}
Let $\varphi\in\mathcal{A}_{\kappa}(\mathbb{R}^{d})$ be a nonnegative function, then $D_{\kappa}(\varphi)$
is Dunkl positive definite.
\end{theorem}

\begin{proof}
For $\varphi\in\mathcal{A}_{\kappa}(\mathbb{R}^{d})$, we have 
\begin{equation*}
\tau_{y}\left(D_{\kappa}(\varphi)\right)(x)=\int_{\mathbb{R}^{d}} E_{\kappa}(-iy,\xi)E_{\kappa}(ix,\xi)\varphi(-\xi)h_{\kappa}^{2}(\xi)d\xi.
\end{equation*}
Thus, 
\begin{equation*}
\begin{split}
\sum_{j=1}^{n}\sum_{l=1}^{n}\alpha_{j}\overline{\alpha_{k}}\tau_{x_{j}}\left(D_{\kappa}(\varphi)\right)(x_{l})&=\int_{\mathbb{R}^{d}}\left[\sum_{j=1}^{n}\sum_{l=1}^{n}\alpha_{j}\overline{\alpha_{l}}\left(E_{\kappa}(-ix_j,\xi)E(ix_l,\xi)\right)\right]\varphi(-\xi)h_{\kappa}^{2}(\xi)d\xi\\
&=\int_{\mathbb{R}^{d}}\left[\sum_{j=1}^{N}\alpha_{j}E_{\kappa}(-ix_{j},\xi)\right]\overline{\left[\sum_{l=1}^{N}\alpha_{l}E_{\kappa}(-ix_{l},\xi)\right]}\varphi(-\xi)h_{\kappa}^{2}(\xi)d\xi\\
&=\int_{\mathbb{R}^{d}}\bigg|\sum_{j=1}^{N}\alpha_{j}E_{\kappa}(-ix_{j},\xi)\bigg|^{2}\varphi(-\xi)h_{\kappa}^{2}(\xi)d\xi\geq0.
\end{split}
\end{equation*}
Which completes the proof.
\end{proof}
\begin{exe}
For $t>0$, the function 
$$F_t(x)=e^{-t\parallel x\parallel^{2}}$$
is Dunkl positive definite.\\
Indeed, put
$$G_t(x)=\frac{c_{\kappa}}{(4t)^{\gamma+\frac{d}{2}}}e^{-\frac{\parallel x\parallel^{2}}{4t}}.$$
Thus, $G_t$ is nonnegative function of $L^{1}\left(\mathbb{R}^{d},\, h_{\kappa}^{2}\right).$
Moreover, (\rm{see [9]})
\begin{equation}\label{e2}
F_t(x)=D_\kappa(G_t)(x).
\end{equation}
Since $F_t(x)\in L^{1}\left(\mathbb{R}^{d},\, h_{\kappa}^{2}\right)$, we conclude by theorem \ref{t2}.
\end{exe}
\begin{exe}
Consider the modified Bessel function of the second kind of order $\alpha$ defined by
\[K_{\alpha}(x)=\int_{0}^{+\infty}e^{-x\cosh(t)}\cosh(\alpha t)dt,\,\,\, x>0.\]
Using the integral representation \rm{[5,(7.12.24]}
\[2K_{\alpha}(ax)=a^{\alpha}\int_{0}^{+\infty}t^{-1-\alpha}e^{-\frac{x}{2}(t+\frac{a^{2}}{t})}dt,\]
by setting $a=r,\;x=1$ and substituting $u=2t$, and using $K_{\alpha}=K_{-\alpha}$, we have 
\begin{equation}\label{e4}
K_{\alpha}(r)=r^{-\alpha}2^{\alpha-1}\int_{0}^{+\infty}u^{\alpha-1}e^{-u}e^{-\frac{r^{2}}{4u}}du,.
\end{equation}
Now, putting
\[\Phi(y)=\frac{1}{(1+\parallel y\parallel_{2}^{2})^{p}},\,\,\, y\in\mathbb{R}^{d},\]
with $p\in\mathbb{N},$ such that $p>\frac{d}{2}+\gamma+1.$\\
Since $p>\frac{d}{2}+\gamma+1,$ the function $\Phi$ is in $(L^{1}\cap L^{2})\left(\mathbb{R}^{d},\, h_{\kappa}^{2}\right)$. From, the integral representation of  the gamma function, for $p>0$, we have
\begin{equation*}
\begin{split}
\Gamma(p)&=\int_{0}^{\infty}t^{p-1}e^{-t}dt\\
&=s^{p}\int_{0}^{\infty}u^{p-1}e^{-su}du.
\end{split}
\end{equation*}
Let $s=1+\parallel y\parallel_{2}^{2}$, then we get
\[\Phi(y)=\frac{1}{\Gamma(p)}\int_{0}^{\infty}u^{p-1}e^{-u}e^{-u\parallel y\parallel_{2}^{2}}du.\]
Thus
\begin{equation*}
\begin{split}
D_{\kappa}(\Phi)(\omega)&=c_{\kappa}\int_{\mathbb{R}^{d}}\Phi(x)E_{\kappa}(x,-i\omega)h_\kappa^{2}(x) dx\\
&=\frac{c_{\kappa}}{\Gamma(p)}\int_{\mathbb{R}^{d}}\int_{0}^{\infty}u^{p-1}e^{-u}e^{-u\parallel y\parallel_{2}^{2}}E_{\kappa}(y,-i\omega)h_{\kappa}^{2}(y)dydu\\
&=\frac{1}{\Gamma(p)}\int_{0}^{\infty}u^{p-1}e^{-u}\left[c_{\kappa}\int_{\mathbb{R}^{d}}e^{-u\parallel y\parallel_{2}^{2}}E_{\kappa}(y,-i\omega)h_{\kappa}^{2}(y)dy\right]du\\
&=\frac{1}{\Gamma(p)}\int_{0}^{\infty}u^{p-1}e^{-u}D_{\kappa}\left(F_{u}(.)\right)(\omega)du\\
&=\frac{c_{\kappa}}{\Gamma(p)2^{\gamma+\frac{d}{2}}}\int_{0}^{\infty}u^{p-\gamma-\frac{d}{2}-1}e^{-u}e^{-\frac{\parallel\omega\parallel_{2}^{2}}{4u}}du
\end{split}
\end{equation*}
Using the relation (\ref{e4}) we obtain
\begin{equation}\label{3}
D_{\kappa}(\Phi)(\omega)=\frac{c_{\kappa}}{2^{p-1}}\parallel\omega\parallel_{2}^{p-\gamma-\frac{d}{2}}K_{p-\gamma-\frac{d}{2}}(\parallel\omega\parallel_{2}).
\end{equation}
Since for $\alpha>0$, the even function $x^\alpha K_\alpha(x)$ is positive and belongs to $L^1([0,+\infty[,x^{2\alpha+1}dx)$, by the inversion formula and theorem \ref{t2} we deduce that $\Phi(y)=\frac{1}{(1+\parallel y\parallel_{2}^{2})^{p}}$ is a Dunkl positive definite function. 
\end{exe}

\begin{exe}\label{p2}
Let $\varphi \in L^{2}\left(\mathbb{R}^{d},\, h_{\kappa}^{2}\right)$ be a continuous function.  We consider the functions $\gamma_t,\;\;t>0$, defined by
\[\gamma_{t}(y)=\sum_{j=1}^{N}\alpha_{j}\tau_{x_{j}}\left(G_{t}(y)\right),\;y\in\mathbb{R}^{d}\]
where $\alpha_j\in\mathbb{C},\;x_{j}\in\mathbb{R}^{d}$ for all $1\leq j\leq N$ and $G_t$ is the function definite in example1.\\ 

\noindent If $<\varphi \star_\kappa \gamma_t,\gamma_t>\geq0$, then $\varphi$ is Dunkl positive definite.
\end{exe}
\noindent Indeed, by the definition of the generalized translation operator and equation (\ref{e2}), we have
\[D_{\kappa}(\gamma_{t})(\omega)=\frac{1}{c_{\kappa}}\sum_{j=1}^{N}\alpha_{j}E_{\kappa}(-ix_{j},\omega)e^{-2t\parallel |\omega\parallel |^{2}},\]
which leads to
\begin{equation}
\begin{split}
D_{\kappa}\left(\gamma_{t}\star_{\kappa}\overline{\gamma^{\vee}}_{t}\right)(\omega)&=\mid D_{\kappa}(\gamma_{t})\mid^{2}(\omega)\\
&=\frac{1}{c_{k}^{2}}\bigg|\sum_{j=1}^{N}\alpha_{j}E_{\kappa}(-ix_{j},\omega)\bigg|^{2}e^{-2t\parallel \omega\parallel^{2}}\\
&=\frac{1}{c_{k}}\sum_{j,l=1}^{N}\alpha_{j}\overline{\alpha_{l}}E_{\kappa}(-ix_{j},\omega)\overline{E_{\kappa}(-ix_{l},\omega)}e^{-2t\parallel \omega \parallel^{2}}\\
&=D_{\kappa}\left(\sum_{j,l=1}^{N}\alpha_{j}\overline{\alpha_{l}}\tau_{x_{j}}\left[\tau_{-x_{l}}G_{2t}(.)\right]\right)(\omega).
\end{split}
\end{equation}
Since the Dunkl transform is a topological automorphism of the Schwartz
space $\mathcal{S}(\mathbb{R}^{d})$, then 
$$\gamma_{t}\star_{\kappa}\overline{\gamma^{\vee}}_{t}(\gamma)= \sum_{j,l=1}^{N}\alpha_{j}\overline{\alpha_{l}}\tau_{x_{j}}\left[\tau_{-x_{l}}G_{2t}(.)\right](\omega),$$
i.e
$$\gamma_{t}\star_{\kappa}\overline{\gamma^{\vee}}_{t}(\omega)= \sum_{j,l=1}^{N}\alpha_{j}\overline{\alpha_{l}}\tau_{x_{j}}\left(\Gamma_{\kappa}(2t,x_l,.)\right)(\omega),$$
where $\Gamma_\kappa$ is the Dunkl type heat kernel.
Thus,
\begin{equation*}
\begin{split}
\int_{\mathbb{R}^{d}}\varphi(y)\,\gamma_{t}\star_{\kappa}\overline{\gamma^{\vee}_{t}}(y)h_{\kappa}^{2}(y)dy&=\sum_{j,l=1}^{N}\alpha_{j}\overline{\alpha_{k}}\int_{\mathbb{R}^{d}}\varphi(y)\tau_{x_{j}}\Gamma_{\kappa}(2t,\, x_{l},\, y)h_{\kappa}^{2}(y)dy,\\
&=\sum_{j,l=1}^{N}\alpha_{j}\overline{\alpha_{l}}\int_{\mathbb{R}^{d}} \tau_{x_{j}}\varphi(y)\,\Gamma_{\kappa}(2t,\, x_{l},\, y)h_{\kappa}^{2}(y)dy.
\end{split}
\end{equation*}
By theorem 4.7 in [9], we have 
\[\lim_{t\longrightarrow0}\int_{\mathbb{R}^{d}}\varphi(y)\gamma_{t}\star_{\kappa}\overline{\gamma^\vee_{t}}(y)h_{\kappa}^{2}(y)dy=\sum_{j,l=1}^{N}\alpha_{j}\overline{\alpha_{l}}\,\,\, \tau_{x_{j}}\varphi(x_{l}).\]
Which completes the proof.

\begin{propo}\label{p1}
Let $\varphi \in\mathcal{A}_{\kappa}(\mathbb{R}^{d})$. If $\varphi$ is Dunkl positive definite function and $f \in L^{2}\left(\mathbb{R}^{d},\, h_{\kappa}^{2}\right)$, then
\begin{equation}\label{e1}
<\varphi \star_\kappa f,f>\geq0.
\end{equation}
\end{propo}

\begin{proof}
Since $\varphi\in\mathcal{A}_{\kappa}(\mathbb{R}^{d}),$ and $f\in L^{2}\left(\mathbb{R}^{d},\, h_{\kappa}^{2}\right),$
then $\varphi\star_{\kappa}f\in L^{2}\left(\mathbb{R}^{d},\, h_{\kappa}^{2}\right),$
and
$$\varphi\star_{\kappa}f(x)=\int_{\mathbb{R}^{d}}\tau_{x}\check{\varphi}(y)f(y)h_{\kappa}^{2}(y)dy.$$
Since $\varphi$ is Dunkl positive definite function, then  $\overline{\varphi}=\check{\varphi}$, so 
 $$\varphi\star_{\kappa}f(x)=\int_{\mathbb{R}^{d}}\tau_{x}(\overline{\varphi})(y)f(y)h_{\kappa}^{2}(y)dy.$$
Thus, for $f \in L^{2}\left(\mathbb{R}^{d},\, h_{\kappa}^{2}\right)$,

$$<\varphi \star_\kappa f,f>=\int_{\mathbb{R}^{d}}\int_{\mathbb{R}^{d}}\tau_{x}(\overline{\varphi})(y)f(y)\overline{f(x)}h_{\kappa}^{2}(y)dy h_{\kappa}^{2}(x)dx.$$
Let $f\in \mathcal{S}(\mathbb{R}^{d})$, its known that, for $\epsilon>0$, there exists a closed cube $W\subseteq\mathbb{R}^{d}$, such that

$$\bigg|\int_{\mathbb{R}^{d}}\int_{\mathbb{R}^{d}}\tau_{x}(\overline{\varphi})(y)f(y)\overline{f(x)}h_{\kappa}^{2}(y)dy h_{\kappa}^{2}(x)dx-\int_{W}\int_{W}\tau_{x}(\overline{\varphi})(y)f(y)\overline{f(x)}h_{\kappa}^{2}(y)dy h_{\kappa}^{2}(x)dx\bigg|<\frac{\epsilon}{2}.$$
But the double integral over the cubes is the limit of Riemannian sums. Hence, we can find $x_{1},\,...,\, x_{N}\in\mathbb{R}^{d}$ and weights $\omega_{1},\,...,\, \omega_{N}$ such that
$$\bigg| \int_{W}\int_{W}\tau_{x}(\overline{\varphi})(y)f(y)\overline{f(x)}h_{\kappa}^{2}(y)dy h_{\kappa}^{2}(x)dx-\sum_{j,l=1}^{N}\tau_{x_{j}}\varphi(x_{l})f(x_{j})\omega_{j}\overline{f(x_{l})\omega_{l}} \bigg|<\frac{\epsilon}{2}.$$
This means that
$$\int_{\mathbb{R}^{d}}\int_{\mathbb{R}^{d}}\tau_{x}(\overline{\varphi})(y)f(y)\overline{f(x)}h_{\kappa}^{2}(y)dy h_{\kappa}^{2}(x)dx>\sum_{j,l=1}^{N}\tau_{x_{j}}\varphi(x_{l})f(x_{j})\omega_{j}\overline{f(x_{l})\omega_{l}}-\epsilon.$$
Letting $\epsilon$ tend to zero and using that $\varphi$ is Dunkl positive definite function shows that (\ref{e1}) is true for all $f\in L^{2}\left(\mathbb{R}^{d},\, h_{\kappa}^{2}\right).$
\end{proof}
\begin{coro}
Let $\varphi\in\mathcal{A}_{\kappa}(\mathbb{R}^{d})$ be a Dunkl positive definite function, we define $\lambda:\,\mathcal{S}(\mathbb{R}^{d})\longrightarrow\mathbb{C}$
by 
\begin{equation}\label{c1}
\lambda(\gamma)=\int_{\mathbb{R}^{d}}\varphi(x)D_{\kappa}^{-1}(\gamma)(x)h_{\kappa}^{2}(x)dx.
\end{equation}
If $\gamma=|\psi|^{2}$ with $\psi\in\mathcal{S}(\mathbb{R}^{d})$
and even, then $\lambda(\gamma)$ is nonnegative.
\end{coro}
\begin{proof}
Put $f=D_\kappa^{-1}(\psi)$. Since $\psi$ is even, then $f$ and $D_\kappa(f)$ are even, and\\ 
$D_\kappa(f)(x)=\psi(x).$\\
Thus,
$$\gamma(x)=\psi(x)\overline{\psi(x)}=D_{\kappa}f(x).\overline{D_{\kappa}f(x)}=D_{\kappa}f(x).D_{\kappa}\overline{f}(x)=D_{\kappa}\left(f\star_{\kappa}\overline{f}\right)(x).$$
Then,
\begin{align*}
\lambda(\gamma)&=\int_{\mathbb{R}^{d}}\varphi(x)D_{\kappa}^{-1}(\gamma)(x)h_{\kappa}^{2}(x)dx\\
&=\int_{\mathbb{R}^{d}}\varphi(x)\left(f\star_{\kappa}\overline{f}\right)(-x)h_{\kappa}^{2}(x)dx\\
&=\int_{\mathbb{R}^{d}}\varphi(x)\left[\int_{\mathbb{R}^{d}}\tau_{-x}f(y)\overline{f(y)}h_{\kappa}^{2}(y)dy\right]h_{\kappa}^{2}(x)dx\\
&=\int_{\mathbb{R}^{d}}\overline{f(y)}\left[\int_{\mathbb{R}^{d}}\varphi(x)\tau_{-x}f(y)h_{\kappa}^{2}(x)dx\right]h_{\kappa}^{2}(y)dy\\
&=\int_{\mathbb{R}^{d}}\overline{f(y)}\left[\int_{\mathbb{R}^{d}}\varphi(x)\tau_{-y}f(x)h_{\kappa}^{2}(x)dx\right]h_{\kappa}^{2}(y)dy\\
&=\int_{\mathbb{R}^{d}}\overline{f(y)}\left(\varphi\star_{\kappa}f\right)(-y)h_{\kappa}^{2}(y)dy\\
&=\int_{\mathbb{R}^{d}}\overline{f(-y)}\left(\varphi\star_{\kappa}f\right)(y)h_{\kappa}^{2}(y)dy\\
&=\int_{\mathbb{R}^{d}}\overline{f(y)}\left(\varphi\star_{\kappa}f\right)(y)h_{\kappa}^{2}(y)dy\\
&=<\varphi\star_{\kappa}f,f>\geq 0.
\end{align*}
\end{proof}
\begin{propo}\label{p2}
Let $\varphi \in\mathcal{A}_{\kappa}(\mathbb{R}^{d})$. If $\varphi$ is a Dunkl positive definite function, then 
$$D_{\kappa}\varphi(x)\geq 0,\;\;\forall x\in \mathbb{R}^{d}.$$
\end{propo}
\begin{proof}
For $\varphi, f\in\mathcal{A}_{\kappa}(\mathbb{R}^{d})$, we have
\begin{equation*}
\begin{split}
\int_{\mathbb{R}^{d}}|D_{\kappa}(f)|^{2}(\xi)D_{\kappa}(\varphi)(\xi)h_{\kappa}^{2}(\xi)d\xi
&=\int_{\mathbb{R}^{d}}D_{\kappa}(f)(\xi)\overline{D_{\kappa}(f)(\xi)}D_{\kappa}(\varphi)(\xi)h_{\kappa}^{2}(\xi)d\xi\\
&=c_{\kappa}^{2}\int_{\mathbb{R}^{d}}D_{\kappa}(\varphi)(\xi)\left[\int_{\mathbb{R}^{d}}f(x)E_{\kappa}(x,-i\xi)h_{\kappa}^{2}(x)dx\right]\\
&\times\left[\int_{\mathbb{R}^{d}}\overline{f(y)}E_{\kappa}(y,\, i\xi)h_{\kappa}^{2}(y)dy\right]h_{\kappa}^{2}(\xi)d\xi\\
&=c_{\kappa}^{2}\int_{\mathbb{R}^{d}}\int_{\mathbb{R}^{d}}f(x)\overline{f(y)}\left[\int_{\mathbb{R}^{d}}E_{\kappa}(x,-i\xi)E_{\kappa}(y,\, i\xi)D_{\kappa}(\varphi)(\xi)h_{\kappa}^{2}(\xi)d\xi\right]\\
&\times h_{\kappa}^{2}(x)dxh_{\kappa}^{2}(y)dy\\
&=c_{\kappa}^{2}\int_{\mathbb{R}^{d}}\int_{\mathbb{R}^{d}}f(x)\overline{f(y)}\tau_{x}\varphi(y)h_{\kappa}^{2}(x)dx h_{\kappa}^{2}(y)dy\\
&=c_{\kappa}^{2} <\overline{\varphi}\star_\kappa f,f >\geq 0.
\end{split}
\end{equation*}
Since $\varphi$ is Dunkl positive definite then $D_{\kappa} \varphi$ is real. Since the last inequality holds for an arbtitrary function $f\in\mathcal{A}_{\kappa}(\mathbb{R}^{d}),$ we conclude.
\end{proof}
\begin{coro}
Let $\varphi \in\mathcal{A}_{\kappa}(\mathbb{R}^{d})$. If $\varphi$ is a Dunkl positive definite function, then $\varphi$ is bounded and 
$$|\varphi(x)|\leq\varphi(0),\;\;\forall x\in\mathbb{R}^{d}.$$
\end{coro}
\begin{proof}
In definition \ref{d1}, let $n=2,\alpha_1=|\varphi(x)|, \alpha_2=-\overline{\varphi(x)}, x_1=0$ and $x_2=x,$ we have 
$$\varphi(0)|\varphi(x)|^{2}-\varphi(-x)|\varphi(x)|\varphi(x)-\overline{\varphi(x)}|\varphi(x)|\varphi(x)+|\varphi(x)|^{2}\tau_{x}\varphi(x)\geq0.$$
Since $\overline{\varphi(-x)}=\varphi(x)$, we obtain

$$|\varphi(x)|^{2}\left[\varphi(0)-2|\varphi(x)|+\tau_{x}\varphi(x)\right]\geq0$$

i.e

\begin{equation}\label{022}
|\varphi(x)|\leq\frac{1}{2}(\varphi(0)+\tau_x\varphi(x)).
\end{equation}

Furthermore, by the definition of Dunkl translation, and since $\varphi$ is Dunkl positive definite function, we have
\begin{equation}\label{023}
\begin{split}
\tau_{x}\varphi(x)=\big|\tau_{x}\varphi(x)\big|&=\bigg|\int_{\mathbb{R}^{d}}\big|E_\kappa(ix,\xi)\big|^{2}D_\kappa\varphi(\xi)h_{\kappa}^{2}(\xi)d\xi\bigg|\\
& \leq\int_{\mathbb{R}^{d}}\big|D_{\kappa}\varphi(\xi)\big|h_{\kappa}^{2}(\xi)d\xi\\
&=\int_{\mathbb{R}^{d}}D_{\kappa}\varphi(\xi)h_{\kappa}^{2}(\xi)d\xi\\
&=\varphi(0).
\end{split}
\end{equation}
The relations (\ref{022}) and (\ref{023}) lead to 

$$|\varphi(x)|\leq\varphi(0).$$
\end{proof}
\begin{coro}
Let $\varphi_{1},\,\varphi_{2}\in\mathcal{A}_{\kappa}(\mathbb{R}^{d}).$ If $\varphi_{1},\,\varphi_{2}$ are Dunkl positive definite functions, then the convolution product $\varphi_{1}\star_{\kappa}\varphi_{2}$ is also.
\end{coro}
\begin{proof}
For $\varphi_{1},\,\varphi_{2}\in\mathcal{A}_{\kappa}(\mathbb{R}^{d})$, we have
$$\varphi_{1}\star_\kappa \varphi_{2}\in L^{1}\left(\mathbb{R}^{d},\, h_{\kappa}^{2}\right),$$
and 
$$D_{\kappa}\left(\varphi_{1}\star_{\kappa}\varphi_{2}\right)=D_{\kappa}\varphi_{1}.D_{\kappa}\varphi_{2}
\in L^{1}\left(\mathbb{R}^{d},\, h_{\kappa}^{2}\right).$$
Now, for every complex numbers $\alpha_{1},\,...\,\,\alpha_{n}$ and for every distinct real numbers $x_{1},\,...\,,\, x_{n}$, we have
\begin{equation*}\label{rr}
\begin{split} 
\sum_{j=1}^{n}\sum_{l=1}^{n}\alpha_{j}\overline{\alpha_{l}}\tau_{x_{j}}\left(\varphi_{1}\star_{\kappa}\varphi_{2}\right)(x_{l})&=\int_{\mathbb{R}^{d}}\sum_{j=1}^{N}\sum_{l=1}^{N}\alpha_{j}\overline{\alpha_{l}}E_{\kappa}\left(ix_{j},\xi\right)E_{\kappa}\left(-ix_{l},\xi\right)D_{\kappa}\left(\varphi_{1}\star_{\kappa}\varphi_{2}\right)(\xi)h_{k}^{2}(\xi)d\xi\\
&=\int_{\mathbb{R}^{d}}\bigg|\sum_{j=1}^{n}\alpha_{j}E_{\kappa}\left(ix_{j},\xi\right)\bigg|^{2}D_{\kappa}\left(\varphi_{1}\right)(\xi)D_{\kappa}\left(\varphi_{2}\right)(\xi)h_{k}^{2}(\xi)d\xi\\
&\geq 0,
\end{split}
\end{equation*}
where the last inequality follows from Proposition \ref{p2}.
\end{proof}

\begin{propo}\label{p5}
Let $\varphi\in\mathcal{A}_{\kappa}(\mathbb{R}^{d})$ be a radial and Dunkl positive definite function. If $f\in\mathcal{A}_{\kappa}(\mathbb{R}^{d})$ is a positive radial function, then the product $\varphi D_\kappa f$ is a Dunkl positive definite function.
\end{propo}

\begin{proof}
Since $\varphi,f\in\mathcal{A}_{\kappa}(\mathbb{R}^{d})$, and radials we have 
$$\varphi D_\kappa f\in\mathcal{A}_{\kappa}(\mathbb{R}^{d}), $$
and radial function.
Thus,
\begin{equation*}
\begin{split}
\sum_{j=1}^{n}\sum_{l=1}^{n}\alpha_{j}\overline{\alpha_{l}}\tau_{x_{j}}(\varphi D_{\kappa}f)(x_{l})&=\int_{\mathbb{R}^{d}}\bigg|\sum_{j=1}^{n}\alpha_{j}E_{\kappa}(ix_{j},\,\xi)\bigg|^{2}D_{\kappa}\left(\varphi D_\kappa f\right)(\xi)h_{\kappa}^{2}(\xi)d\xi\\
&=\int_{\mathbb{R}^{d}}\bigg|\sum_{j=1}^{n}\alpha_{j}E_{\kappa}(ix_{j},\,\xi)\bigg|^{2}D_{\kappa}\varphi\star_{\kappa}f(\xi)h_{\kappa}^{2}(\xi)d\xi.
\end{split}
\end{equation*}
Moreover, by the definition of Dunkl convolution, we can write
\begin{equation}{\label{aa}}
D_{\kappa}\varphi\star_{\kappa}f(x)=\int_{\mathbb{R}^{d}}D_{\kappa}\varphi(t)\tau_{x}\check{f}(t)h_{\kappa}^{2}(t)dt.
\end{equation}
From proposition \ref{p2} and theorem 3.4 in [10], we have 
$$D_{\kappa}\varphi\star_{\kappa}f(x)\geq0.$$
Which completes the proof.
\end{proof}
\begin{coro}
Let $\varphi_{1},\varphi_{2}\in\mathcal{A}_{\kappa}(\mathbb{R}^{d})$ are radials. If $\varphi_{1},\varphi_{2}$ are Dunkl positive definite  functions,
then the product $\varphi_{1}\varphi_{2}$ is also.
\end{coro}
\begin{proof}
Let $\psi=D_{\kappa}\varphi_{2}$, then $D_{\kappa}\psi=\varphi_{2},$ and since $\varphi_{2}$ is radial, we have $\psi$ is radial.
So, by proposition \ref{p2}, we have 
$$\psi\geq 0.$$  
By proposition \ref{p5}, we conclude.
\end{proof}

\noindent In the following we state a version of  Bochner's theorem in Dunkl setting
and we establish a necessary and sufficient condition for a function to be a Dunkl
positive definite.
\begin{theorem}(Bochner)\label{boc}
Let $\varphi\in \mathcal{A}_{\kappa}(\mathbb{R}^{d}).$ Then, $\varphi$ is Dunkl positive
definite, if and only if, there exist a nonnegative function $\psi\in \mathcal{A}_{\kappa}(\mathbb{R}^{d})$ such that
\begin{equation}
\varphi=D_{\kappa}\psi.
\end{equation}
\end{theorem}
\begin{proof}
Since $\varphi$ is Dunkl positive definite function, we have $\overline{\varphi(-x)}=\varphi(x),$ and $D_{\kappa}\varphi$ is even real function (see corollary \ref{c11}). By the inversion formula we have
$$\varphi(x)=D_{\kappa}^{2}\varphi(x),\;\;\forall x\in\mathbb{R}^{d}.$$
Let
$$\psi(x)=D_\kappa\varphi(x).$$
By proposition \ref{p2}, we deduce that $\psi$ is nonnegative function of $\mathcal{A}_{\kappa}(\mathbb{R}^{d}).$\\
Inversely, since $\psi$ is nonnegative function and belong to $\mathcal{A}_{\kappa}(\mathbb{R}^{d}),$ by theorem \ref{t2} we deduce that $\varphi=D_\kappa\psi$ is Dunkl positive definite function.
\end{proof}
\subsection{Applications}
\begin{propo}\label{p4}
Let $\varphi\in\mathcal{A}_{\kappa}(\mathbb{R}^{d})$ be a radial function. If $\varphi$ is Dunkl positive definite function, then there exist a nonnegative  radial function  $\psi\in\mathcal{A}_{\kappa}(\mathbb{R}^{d})$ such that
\begin{enumerate}
\vspace{0.1cm}
\item $\displaystyle{\tau_{y}\psi\geq0,}$
\vspace{0.2cm}
\item $\displaystyle{\tau_{y}\psi\in L^{1}\left(\mathbb{R}^{d},h_{\kappa}^{2}\right),}$
\end{enumerate}
and
\[\parallel\tau_{y}\psi\parallel_{1,\kappa}=\parallel\psi\parallel_{1,\kappa}=\varphi(0).\]
\end{propo}

\begin{proof}
Bochner's theorem asserts that there exist a nonegative function $\psi$ such that 
$$\varphi=D_{\kappa}\psi.$$
Since $\varphi$ is radial, then $\psi=D_{\kappa}\varphi$ is radial, nonnegative and belongs to $\mathcal{A}_{\kappa}(\mathbb{R}^{d})$.\\
Using theorem 3.4 in [10], we get \\
$$\;\;\textrm{(i)}\;\tau_{y}\psi\geq0,\;\;\tau_{y}\psi\in L^{1}\left(\mathbb{R}^{d},h_{\kappa}^{2}\right).$$
$$\textrm{(ii)}\parallel\tau_{y}\psi\parallel_{1,\kappa}=\parallel\psi\parallel_{1,\kappa}=\varphi(0).$$
\end{proof}
\begin{coro}\label{c4}
For $t>0$ and $x\in\mathbb{R}^{d}$, we have
$$\Gamma_{\kappa}(t,x,y)\geq0,\;\forall{y\in\mathbb{R}^{d}}$$
$$\;\;\Gamma_{\kappa}(t,x,y)\in L^{1}\left(\mathbb{R}^{d},h_{\kappa}^{2}(y)\right),$$
and
\begin{equation}\label{025}
\int_{\mathbb{R}^{d}}\Gamma_{\kappa}(t,x,y)h_{\kappa}^{2}(y)dy=1.
\end{equation}
Where $\Gamma_{\kappa}(t,x,y)$ is the Dunkl type heat kernel defined by 
\[\displaystyle{\Gamma_{\kappa}(t,x,y)=\frac{c_{\kappa}}{(4t)^{\gamma+\frac{d}{2}}}e^{-\left(\frac{\parallel x\parallel^{2}+\parallel y\parallel^{2}}{4t}\right)}E_{\kappa}\left(\frac{x}{\sqrt{2t}},\frac{y}{\sqrt{2t}}\right).}\]
\end{coro}
\begin{proof} For $t>0,$ the function $\varphi(x)=e^{-t\parallel x\parallel^{2}}=F_{t}(x)$
is Dunkl positive definite function (see example 1), radial and belongs
to $\mathcal{A}_{\kappa}(\mathbb{R}^{d}).$
Moreover, 
\[\psi(x)=D_{\kappa}(\varphi)(x)=\frac{c_{\kappa}}{(4t)^{\gamma+\frac{d}{2}}}e^{-\frac{\parallel x\parallel^{2}}{4t}}.\]
Then
\[\Gamma_{\kappa}(t,x,y)=\tau_{x}(\psi)(y)=\tau_{x}(D_{\kappa}\varphi)(y).\]
By the last proposition $\Gamma_{\kappa}(t,x,y)$ is nonnegative, belongs to  $L^{1}\left(\mathbb{R}^{d},h_{\kappa}^{2}\right),$  and we have 
\begin{equation*}
\begin{split} 
\int_{\mathbb{R}^{d}}\Gamma_{\kappa}(t,x,y)h_{\kappa}^{2}(y)dy&=\int_{\mathbb{R}^{d}}\tau_{x}(\psi)(y)h_{\kappa}^{2}(y)dy\\
&=\int_{\mathbb{R}^{d}}\psi(y)h_{\kappa}^{2}(y)dy\\
&=\int_{\mathbb{R}^{d}}D_{\kappa}\varphi(y)h_{\kappa}^{2}(y)dy\\
&=\varphi(0)=1.
\end{split}
\end{equation*}
\end{proof}
\begin{coro}
For $p\geq \gamma+\frac{d}{2}+1$, let $K_{\alpha}$ be the modified Bessel of the second kind and order $\alpha$, then
\begin{equation*}
\displaystyle{\int_{\mathbb{R}^{d}}\tau_{y}\left(\parallel x\parallel^{p-\gamma-\frac{d}{2}}K_{p-\gamma-\frac{d}{2}}(\parallel x\parallel)\right)(\xi)h_{\kappa}^2(\xi)d\xi=\int_{\mathbb{R}^{d}}\parallel x\parallel^{p-\gamma-\frac{d}{2}}K_{p-\gamma-\frac{d}{2}}(\parallel x\parallel)h_{\kappa}^2(x)dx=\frac{\Gamma(p)}{c_{\kappa}2^{p-1}}.}
\end{equation*}
\end{coro}
\begin{proof} Let $p\geq \gamma+\frac{d}{2}+1$, be an integer. Put $\varphi(y)=\frac{1}{\left(1+\parallel y\parallel^{2}\right)^{p}}$, then $\varphi$  is Dunkl positive definite function (see example 2), we have 
\[\psi(\lambda)=D_{\kappa}(\varphi)(\lambda)=\frac{c_{\kappa}}{\Gamma(p)2^{p-1}}\parallel\lambda\parallel^{p-\gamma-\frac{d}{2}}K_{p-\gamma-\frac{d}{2}}(\parallel\lambda\parallel).\]
By the last proposition we have
\begin{equation*}
\begin{split} 
\int_{\mathbb{R}^{d}}\tau_{x}(\psi)(y)h_{\kappa}^{2}(y)dy&=\int_{\mathbb{R}^{d}}\psi(y)h_{\kappa}^{2}(y)dy\\
&=\int_{\mathbb{R}^{d}}D_{\kappa}\varphi(y)h_{\kappa}^{2}(y)dy\\
&=\varphi(0)=1.
\end{split}
\end{equation*}
\end{proof}
\section{Strictly Dunkl Positive Definite Functions}
\begin{lemme}\label{l1}
Let $U\subseteq\mathbb{R}^{d}$ is open. Suppose that $x_{1},\,...\,,x_{n}\in\mathbb{R}^{d}$, are pairwise distinct
and that $\alpha=(\alpha_{1},...,\alpha_{n})\in\mathbb{C}^{n}$. If
$\displaystyle{\sum_{j=1}^{n}\alpha_{j}E_{\kappa}\left(ix_{j},\omega\right)=0}$,
for all $\omega\in U$, then $\alpha\equiv0.$ 
\end{lemme}
\begin{proof}
Suppose that $$\sum_{j=1}^{n}\alpha_{j}E_{\kappa}\left(ix_{j},\omega\right)=0,\;\forall\omega\in U.$$
Since $z\longrightarrow E_{\kappa}(y,z)$ is analytic on $\mathbb{C}$,
by analytic continuation, we get 
$$\sum_{j=1}^{n}\alpha_{j}E_{\kappa}\left(ix_{j},\omega\right)=0,\;\forall\omega\in \mathbb{R}^{d}.$$
Let $f$ be a $C^{\infty}$ function with compactly supported, we know that
\[D_{\kappa}\left(\tau_{x}f\right)(\lambda)=E_{\kappa}(-ix,\lambda)D_{\kappa}f(\lambda).\]
Then
\begin{equation*}
\sum_{j=1}^{n}\alpha_{j}E_{\kappa}\left(ix_{j},\omega\right)=D_{\kappa}\left(\sum_{j=1}^{n}\alpha_{j}\tau_{x_{j}}f\right)(\lambda)=0.
\end{equation*}
Since for all $j\in\left\{ 1,...,n\right\},\;\;\tau_{x_{i}}$ is $C^{\infty}$ function with compactly
supported, then we get
\begin{equation}\label{e9}
\sum_{j=1}^{n}\alpha_{j}\tau_{x_{j}}f(\lambda)=0,\,\,\forall\lambda\in\mathbb{R}^{d}.
\end{equation}
If the support of $f$ is conatained in the ball around zero with radius \\
$\epsilon<\displaystyle{\min_{j\neq k}\big|\parallel x_{k}\parallel-\parallel x_{j}\parallel\big|}$, we have (\rm{see [10]} proposition 3.13), $\tau_{x_{j}}f$ is  supported in $\left\{ x,\,\parallel x\parallel\leq\epsilon+\parallel x_{j}\parallel\right\}.$\\
Thus
$$\tau_{x_{j}}f(x_{k})=0,\;\;\forall\;{k\neq j};\;\tau_{x_{j}}f(x_{j})\neq 0,\; \forall{j,k\in\left\{ 1,...,n\right\}}$$
Using (\ref{e9}), we obtain 
$$\alpha_{j}\tau_{x_{j}}f(x_{j})=0,\;\;\forall \;j\in\left\{ 1,...,n\right\}.$$
We coclude.
\end{proof}
\begin{theorem}
Let $\varphi\in\mathcal{A}_{\kappa}(\mathbb{R}^{d})$, be a nonidentically zero and Dunkl positive definite function. Then $\varphi$ is Dunkl strictly positive definite.
\end{theorem}
\begin{proof}
Let $\varphi\in\mathcal{A}_{\kappa}(\mathbb{R}^{d})$ be nonidentically zero 
and Dunkl positive definite function. Suppose that there exist distinct reals points $x_{1},...,x_{n}$ and complex numbers
$\alpha_{1},...,\alpha_{n}$ not all zero, such that 
\[\sum_{j=1}^{n}\sum_{l=1}^{n}\alpha_{j}\overline{\alpha_{l}}\tau_{x_{j}}\varphi(x_{l})=0.\]
By (11), we get 
\[\int_{\mathbb{R}^{d}}\bigg|\sum_{j=1}^{n}\alpha_{j}E_{\kappa}(-ix_{j},\xi)\bigg|^{2}D_{\kappa}\varphi(\xi)h_{\kappa}^{2}(\xi)d\xi=0.\]
Since $\varphi$ is Dunkl positive definite and belongs to $\mathcal{A}_{\kappa}(\mathbb{R}^{d})$,
we have $D_{\kappa}\varphi$ is nonnegative continuous function.
Then
\[\bigg|\sum_{j=1}^{n}\alpha_{j}E_{\kappa}(-ix_{j},\xi)\bigg|^{2}D_{\kappa}\varphi(\xi)=0,\;\forall\xi\in\mathbb{R}^{d}.\]
Since $D_{\kappa}\varphi$ is nonidentically zero, then there exist an open $U\subset\mathbb{R}^{d}$ such that 
$$D_{\kappa}\varphi(\xi)\neq 0,\;\forall{\xi\in U}.$$  
Thus
$$\bigg|\sum_{j=1}^{n}\alpha_{j}E_{\kappa}(-ix_{j},\xi)\bigg|=0,\forall\xi\in U.$$
Using, lemma \ref{l1}, we get 
$$\alpha_{j}=0,\;\forall j\in\left\{ 1,...,n\right\}.$$
We conclude.
\end{proof}
\begin{exe}
The functions $\varphi(x)=e^{-t\parallel x\parallel^{2}},\;t>0$, and $\psi(x)=\frac{1}{\left(1+\parallel x\parallel^{2}\right)^{p}},\; p\geq\gamma+\frac{d}{2}+1$
are Dunkl strictly positive definite functions.
\end{exe}

\end{document}